\documentclass[11pt]{amsart}

\usepackage{tikz-cd} 
\usepackage[left=25mm,right=25mm, top=25mm, bottom=25mm]{geometry}
\usepackage[english]{babel}
\usepackage[T1]{fontenc}
\usepackage{amsmath}
\usepackage{amstext}
\usepackage{amsfonts}
\usepackage{amssymb}
\usepackage{amsthm}
\usepackage{mathtools}
\usepackage{dsfont}
\usepackage{color}
\usepackage{graphicx}
\usepackage[colorinlistoftodos]{todonotes}
\usepackage[colorlinks=true, allcolors=blue]{hyperref}
\usepackage{float}
\usepackage[colorinlistoftodos]{todonotes}
\usepackage{theoremref}
\usepackage{enumitem}
\usepackage{dirtytalk}
\usepackage{todonotes}
\usepackage{marginnote} 

\usepackage{multicol} 
\usepackage{titletoc}
\usepackage{amsmath}
\usepackage{etoolbox}
\usepackage{hyperref}

\newcounter{numrellocal}
\renewcommand{\thenumrellocal}{\roman{numrellocal}}
\newcounter{numrelglobal}

\makeatletter
\newcommand{\numrel}[2]{
  \stepcounter{numrellocal}
  \refstepcounter{numrelglobal}
  \ltx@label{#2}
  \overset{(\thenumrellocal)}{#1}
}
\makeatother

\allowdisplaybreaks
\let\oldaddcontentsline\addcontentsline

\newcommand{\starttocentries}{\let\addcontentsline\oldaddcontentsline}
\newtheorem{theorem}{Theorem}[section]
\newtheorem{lemma}[theorem]{Lemma}
\newtheorem{prop}[theorem]{Proposition}
\newtheorem{cor}[theorem]{Corollary}

\newtheorem*{cor*}{Corollary}
\newtheorem*{conjecture*}{Conjecture}
\newtheorem*{thm*}{Theorem}
\newtheorem*{lem*}{Lemma}

\newtheorem*{prop*}{Proposition}
\theoremstyle{definition}
\newtheorem{definition}[theorem]{Definition}

\newtheorem{example}[theorem]{Example}
\newtheorem*{defn*}{Definition}

\theoremstyle{remark}

\newcommand{\act}{\curvearrowright}

\DeclareMathOperator{\SL}{SL}

\makeatletter \def\l@subsection{\@tocline{2}{0pt}{1pc}{5pc}{}} \def\l@subsection{\@tocline{2}{0pt}{2pc}{6pc}{}} \makeatother

\title[Relative Commutants inside Tracial Crossed Products]{On relative commutants of subalgebras in group and tracial crossed product von Neumann algebras}
\author[Amrutam]{Tattwamasi Amrutam}
\address{Ben Gurion University of the Negev.
	Department of Mathematics.
	Be'er Sheva, 8410501, Israel.
}
\email{tattwamasiamrutam@gmail.com}
\author[Bassi]{Jacopo Bassi}
\email{jcpbassi123@gmail.com}
\date{\today}

\begin{document}
\begin{abstract} Let $\Gamma$ be a discrete group acting on a compact Hausdorff space $X$. Given $x\in X$, and $\mu\in\text{Prob}(X)$, we introduce the notion of contraction of $\mu$ towards $x$ with respect to unitary elements of a group von Neumann algebra not necessarily coming from group elements. Using this notion, we study relative commutants of subalgebras in tracial crossed product von Neumann algebras. The results are applied to negatively curved groups and $\SL(d,\mathbb{Z})$, $d \geq 2$.
\end{abstract}
\maketitle
\tableofcontents
\section{Introduction}
Operator algebras associated with discrete groups, or more generally discrete group actions, reveal essential properties of the underlying group. Probably the first evidence of this connection is that amenability has neat characterizations at the operator-algebraic level: injectivity of the group von Neumann algebra and nuclearity of the reduced group $C^*$-algebra. Nowadays, it is known that many other group properties have analog descriptions in terms of group $C^*$-algebras, for example, a-T-menability and property T (\cite{BrGu}). Also, free groups' full and reduced $C^*$-algebras can detect their order. It is a significant open problem raised by A. Connes whether non-isomorphic ICC property T groups have
non-isomorphic von Neumann algebras (see \cite{Chifan2023} for examples of ICC property (T) groups with non-isomorphic von Neumann algebras).

Dynamical systems represent a powerful tool for the study of rigidity properties of groups. As an example, rigidity results for certain discrete subgroups of $\SL(2,\mathbb{R})$ can be obtained by looking at the $C^*$-crossed products associated with certain actions (cf. \cite{Ba1, Ba2, GoLi}). Among the possible dynamical systems, an important role is played by boundary actions; for example, the topological amenability of the left action of a discrete group $\Gamma$ on its Stone-{\v C}ech boundary $\partial_\beta \Gamma$ is equivalent to exactness and topological amenability of the left-right action on the same space (which is usually referred to as bi-exactness or property $\mathcal{S}$) implies the Akemann-Ostrand (AO) property (\cite{AnDa}), (i.e., temperedness of the representation of $\Gamma \times \Gamma$ on the Calkin algebra of $l^2 (\Gamma))$, which ensures solidity (hence primeness) of the group von Neumann algebra. A significant open problem in the theory is deciding whether these three properties coincide (cf. \cite{Ba3}).

The solidity of the group von Neumann algebra is a very rigid property, which, for example, captures to some extent the dimension of the ambient group in the case of a lattice in a simple Lie group: it is automatic for discrete subgroups of simple rank-1 Lie groups, and it is automatically
denied by the existence of infinite subgroups with non-amenable centralizers. Weakenings of the (AO) property have been considered in the literature and lead to the definition of properly proximal groups (\cite{boutonnet2021properly}), for which some weaker rigidity properties hold as well. More recently, the notion of biexact von Neumann algebra was introduced in \cite{ding2023biexact}, where examples of von Neumann algebras that are solid but not biexact were given. However, it is still not known if there are non-biexact groups that give rise to solid von Neumann algebras.

 One of the most significant breakthroughs of recent years is the recognition of the central role of proximality arguments in the study of rigidity properties of discrete groups through the lens of dynamical systems, which lead, for example, to the identification of the Furstenberg boundary of a discrete group with the equivariant Hamana-injective envelope of the complex numbers (cf., for example, \cite{kalantar_kennedy_boundaries, Oz, BaRa1}).
 
Given a probability measure, $\mu\in \text{Prob}(X)$, proximality of $\mu$ is nothing but the contraction of this measure with a specific sequence of group elements. In addition, if these group elements leave every finite subset of the group $\Gamma$, then it can be shown that the corresponding unitary elements converge to zero weakly. We generalize this notion to the context of general unitary elements inside the group von Neumann algebra, which does not necessarily come from group elements.
\begin{definition} 
Let $\Gamma$ be a discrete countable group acting on a compact Hausdorff space $X$ endowed with a probability measure $\mu$. Let $(u_n)$ be a sequence of unitaries in $L\Gamma$. We say that $u_n$ is {\it $\mu$-contracting} (towards a point $x \in X$) if for every $\epsilon >0$, every $F \subset \Gamma$ finite and for every $f \in C(X)$ there is $N$ such that for every $n \geq N$ we have $\|u_n|_{A}\|_2 > 1-\epsilon$, where $A= \{ \gamma \in \Gamma \; | \; |\gamma^{-1} \eta \mu (f) - f(x)| < \epsilon\ \; \forall \eta \in F\}$.
    \end{definition}
If $u_n$'s come from the group elements, then the notion of $\mu$-contraction agrees with that of the $\Gamma$-contraction.
Motivated by the notion of solid von Neumann algebra, we employ proximality arguments to study the position of relative commutants of subalgebras in group (and more generally crossed product) von Neumann algebras.
 \begin{theorem}
     Let $\Gamma$ be a discrete countable group acting on a compact Hausdorff space $X$. Let $\{u_n\}\subset L(
    \Gamma)$ be a $\mu$-contracting sequence for the action on $X$ for some probability measure $\mu$ on $X$. Let $(\mathcal{N},\tilde{\tau})$ be a tracial von Neumann algebra, and $\Gamma\curvearrowright(\mathcal{N},\tilde{\tau})$ be a trace-preserving action. 
 Then $\{u_n:n\in\mathbb{N}\}' \cap (\mathcal{N}\rtimes\Gamma) \subset \mathcal{N}\rtimes\Gamma_x$,  where $x$ is determined by the fact that $(u_n)$ is $\mu$-contracting towards $x$, and $\Gamma_x=\{s\in\Gamma: sx=x\}$. 
 \end{theorem}
 
 \subsection{Organization of the paper} Apart from this section, there are four other sections. We prove some preliminary technicalities in Section~\ref{sec:prelim}. In Section \ref{tvna}, using a suitable notion of convergence of a measure to a point under a sequence of unitaries, inspired by \cite{boutonnet2021properly}, we show that the commutant of certain subalgebras of tracial crossed product algebras is contained in the von Neumann algebra associated to the stabilizer of the limiting point. In Section \ref{ncv}, we consider the case of "negatively curved groups" and show that in this case, the position of the relative commutants of subgroup algebras is reminiscent of an averaging property, in the spirit of the Powers' averaging property (cf. for example \cite{haagerup2016new, amrutam2022generalized, Ro}). In Section \ref{sl}, we consider the case of infinite subgroups of $\SL(d,\mathbb{Z})$, in which case we prove that the position of the relative commutant of a subgroup algebra depends on the dynamics of the subgroup in a particular partial flag. Per the results appearing in \cite{BaRa3}, we also see that a weak form of solidity holds for $\SL(3,\mathbb{Z})$.
 
 The authors believe that the techniques developed in this manuscript should have a deep connection with the approach appearing in \cite{BeKa} and \cite{BaRa2}. This connection will be investigated in future work.
\subsection{Acknowledgement} The first named author thanks Yair Glasner and Ido Grayevsky for many helpful discussions around hyperbolic groups. The authors thank Yongle Jiang for reading through a near complete draft of this paper and for his numerous suggestions and corrections. We thank the anonymous referees for their comments and suggestions which improved the readability of the paper. The authors are also supported by research funding from the European Research Council
(ERC) under the European Union's Seventh Framework Program
(FP7-2007-2013) (Grant agreement No. 101078193). The second named author acknowledges the support of INdAM-GNAMPA and the grant "Operator Algebras and Quantum Mathematics (OAQM)," CUP: E83C22001800005.
\section{Preliminaries and Technicalities}
\label{sec:prelim}
Let $\Gamma$ be a discrete group. By a $\Gamma$-space $X$, (also denoted as $\Gamma \act X$ sometimes) on a compact Hausdorff space $X$, we mean a group homomorphism $\pi: \Gamma \to \text{Homeo}(X)$. We often abuse the notation by ignoring $\pi$ and write $s x$ instead of $\pi(s) x$ for $s\in \Gamma$ and $x\in X$.
\begin{definition}
\label{def:contractionforgroupelements}
Let $\Gamma$ be a discrete countable group acting on a second countable compact Hausdorff space $X$ endowed with a probability measure $\mu$. A sequence $(\gamma_n) \subset \Gamma$ is said to be a {\it $\mu$-pointwise contracting sequence} if there is $x \in X$ such that $\gamma_n y \rightarrow x$ for $\mu$-almost every $y \in X$. In this case we say that $(\gamma_n)$ is $\mu$-pointwise-contracting towards $x$.
\end{definition}
If $\mu$ is point-wise contracted by the sequence $\{\gamma_n\}$, then it is also pointwise contracted by $\{\gamma_n\gamma\}$ for any group element $\gamma\in\Gamma$. Moreover, the point where it converges is also unchanged. In other words, the contracting sequence is invariant with respect to the right multiplication by the group elements. We make this precise below.
\begin{lemma}\label{lem1}
    Let $\Gamma$ be a discrete countable group acting on a second countable compact space $X$ and $\mu$ a $\Gamma$-quasi invariant probability measure on $X$. Suppose that there is a $\mu$-pointwise-contracting sequence $(\gamma_n) \subset \Gamma$ towards a point $x \in X$. Then for every $\gamma \in \Gamma$ we have that $\lim_n \gamma_n \gamma \mu =\delta_x$ exists and is independent of $\gamma$.
\end{lemma}
\begin{proof}
    There is $x \in X$ such that for $\mu$-almost every $y \in X$ we have $\lim_n \gamma_n y \rightarrow x$. Let $E \subset X$ be the subset of $\mu$-measure $0$ such that $\lim_n \gamma_n y = x$ for every $y \notin E$. Since $\mu$ is quasi invariant we have that for every $\gamma \in \Gamma$ there is a subset $E'$ of measure zero, namely $\gamma^{-1} E$, such that $\lim_n \gamma_n \gamma y = x$ for every $y \notin E'$. Hence, for every $f \in C(X)$ and every $\gamma \in \Gamma$ we have $f(\gamma_n \gamma y) \rightarrow f(x)$ $\mu$-almost everywhere. It follows from Lebesgue dominated convergence Theorem that $\lim_n \gamma_n \gamma \mu = \lim_n \gamma_n \mu$ in the weak$^*$-topology.
\end{proof}

\subsection{Group von Neumann algebra}
We briefly recall the construction of the group von Neumann algebra. Let $\ell^2(\Gamma)$ be the space of square summable $\mathbb{C}$-valued functions on $\Gamma$. There is a natural action $\Gamma\curvearrowright \ell^2(\Gamma)$ by left translation:
\[\lambda_g\xi(h):=\xi(g^{-1}h), \xi \in \ell^2(\Gamma), g,h \in \Gamma\]
The group von Neumann algebra $L(\Gamma)$ is generated (as a von Neumann algebra inside $\mathbb{B}(\ell^2(\Gamma)$), by the left regular representation $\lambda$ of $\Gamma$. The group von Neumann algebra $L(\Gamma)$ comes equipped with a canonical trace $\tau_0:L(\Gamma)\to\mathbb{C}$ defined by 
\[\tau_0\left(\lambda_g\right)=\left\{ \begin{array}{ll}
0 & \mbox{if $g\ne e$}\\
1 & \mbox{if $g=e$}\end{array}
\right\}\]
It is worth noting that a natural embedding of $L(\Gamma)$ into $\ell^2(\Gamma)$ exists via the map $x\mapsto x\delta_e$. Therefore, any element $x\in L(\Gamma)$ can be expressed as $x=\sum_{g\in \Gamma}x_g\lambda(g)$, where $\lambda(g)\in L(\Gamma)$ correspond to the canonical unitaries of $L(\Gamma)$ and $x_g=\tau_0(x\lambda(g)^*)$ are the Fourier coefficients of $x$. The above sum converges in $\ell^2$-norm ($\|\cdot\|_2$) and not with respect to the strong operator or weak operator topology, as mentioned in \cite[Remark 1.3.7]{ADP}. This expansion is commonly referred to as the Fourier expansion of $x$.

We will make use of the following notion, which generalizes Definition \ref{def:contractionforgroupelements} to sequences of unitaries in a group von Neumann algebra, which do not come from group elements.
\begin{definition} 
\label{def:contraction}
Let $\Gamma$ be a discrete countable group acting on a compact Hausdorff space $X$ endowed with a probability measure $\mu$. Let $(u_n)$ be a sequence of unitaries in $L\Gamma$. We say that $u_n$ is {\it $\mu$-contracting} (towards a point $x \in X$) if for every $\epsilon >0$, every $F \subset \Gamma$ finite and for every $f \in C(X)$ there is $N$ such that for every $n \geq N$ we have $\|u_n|_{A}\|_2 > 1-\epsilon$, where $A= \{ \gamma \in \Gamma \; | \; |\gamma^{-1} \eta \mu (f) - f(x)| < \epsilon\ \; \forall \eta \in F\}$.
    \end{definition}

\begin{lemma} \label{lemb}
    Let $\Gamma$ be a discrete countable group acting on a second countable compact space $X$ and $\mu$ a $\Gamma$-quasi invariant probability measure on $X$. Let $\Lambda \subset \Gamma$ be a subgroup with the property that there is a point $x \in X$ such that every diverging sequence $\lambda_n$ in $\Lambda$ is $\mu$-pointwise-contracting towards $x$. Then every sequence of unitaries in $L\Lambda$ which goes to zero weakly is $\mu$-contracting towards $x$.
\end{lemma}
\begin{proof}
    Let $\epsilon >0$. Since every diverging sequence in $\Lambda$ is $\mu$-contracting towards $x$, it follows from Lemma \ref{lem1} that for every finite set $G \subset \Gamma$ and every $f \in C(X)$ there is a finite set $F \subset \Lambda$ satisfying $|\lambda^{-1} \gamma \mu (f) - f(x) |< \epsilon$ for every $\lambda \in \Lambda \backslash F$, $\gamma \in G$. Now, since $u_n \rightarrow 0$ weakly, there is $N \in \mathbb{N}$ such that $\| u_n|_F \|_2 < \epsilon$ for every $n >N$. The result follows.
\end{proof}
A group $\Gamma$ is called a convergence group if it admits an action $\Gamma\curvearrowright X$ such that for every distinct sequence of elements $\{g_n\}\subset\Gamma$, we can find two elements $a,b\in X$ and a sub-sequence $g_{n_k}$ such that $g_{n_k}|_{X\setminus\{b\}}\to a$ uniformly on every compact subset of $X\setminus \{b\}$ (see \cite{Tukia}). In this case, $a$ is called the attracting point, and $b$ is the repelling point. An element $s$ in a convergence group $\Gamma$ is called parabolic if $s$ has exactly one fixed point on $X$. Moreover, an element $s$ in a convergence group $\Gamma$ is called loxodromic if it has exactly two fixed points on $X$, denoted by $x_s^+$ and $x_s^-$. Moreover, $x_s^+$ is the attractive point for $s$, and $x_s^-$, the repelling point (see \cite[Lemma~2D]{Tukia}).
\begin{example}
\label{ex:convggrp}
Let $\Gamma$ be a convergence group, and $s\in\Gamma$ a parabolic element. Let us call it $x_s^+$. Let $\Lambda=\langle s \rangle$. Using \cite[Lemma~2F]{Tukia}, we see that $\{s^n\}_{n\in\mathbb{Z}}$ is a convergence sequence with the attractive and repelling point of $\{s^n\}_{n\in\mathbb{Z}}$ the same as that of $\{x_s^+\}$. Let us assume that $\Gamma$ is non-elementary, i.e., the set of limit points $LX$ (the collection of all attracting points on $X$) has more than two points. Using \cite[Theorem~2S]{Tukia}, we see that $LX$ is an infinite perfect set. Let $\mu$ be a $\Gamma$-quasi invariant probability measure on $X$ such that $\mu(x_s^+)\ne 0$. It follows from definition~\ref{def:contractionforgroupelements} that every diverging sequence $\lambda_n$ in $\Lambda$ is $\mu$-pointwise-contracting towards $x_s^+$. Let $\mathcal{M}\le L(\Lambda)$ be a diffuse von Neumann subalgebra. Let $u_n\in\mathcal{M}$ be a sequence of unitaries in $\mathcal{M}$ which converges to $0$ weakly. It follows from Lemma~\ref{lemb} that $u_n$ is $\mu$-contracting towards $x_s^+$. Similarly, if $s \in \Gamma$ is loxodromic, then there is a quasi-invariant probability measure $\mu$ on $X$ and points $x_s^{+}$, $x_s^{-}$ such that the sequence $(s^n)$ is $\mu$-pointwise-contracting towards $x_s^{+}$ and $(s^{-n})$ is $\mu$-pointwise-contracting towards $x_s^{-1}$.
\end{example}
Given a $\Gamma$-space $X$, and $\mu\in\text{Prob}(X)$, the Poisson transformation $P_{\mu}:C(X)\to\mathbb{B}(\ell^2(\Gamma))$ is defined by $P_{\mu}(f)(\delta_t)=\mu(t^{-1}f)\delta_t, t\in\Gamma$. It is well-known that $P_{\mu}$ is a $\Gamma$-equivariant unital positive map. Whenever $\mu$ can be contracted using the unitaries, $P_{\mu}$ satisfies a kind of singularity phenomenon.   
\begin{lemma}
\thlabel{SOTconvergence}
Let $\Gamma$ be a discrete countable group acting on a compact Hausdorff space $X$. Let $\mu\in\text{Prob}(X)$. Let $\{u_n\}_n$ be a sequence of unitaries in $L(\Gamma)$ be such that $u_n\to 0$ weakly and is $\mu$-contracting towards $x$. Then, $u_nP_{\mu}(f)u_n^*\xrightarrow[]{\text{SOT}}f(x)$ for every $f\in C(X)$.
\begin{proof}
Let $\{u_n\}_n\in L(\Gamma)$ be a sequence of unitaries satisfying the above assumptions. Let us write $u_n=\sum_{t\in\Gamma}u_n(t)\lambda(t)$, where $u_n(t)=\tau_0(u_n\lambda(t)^*)$ for each $t\in \Gamma$. Moreover, the convergence of the above series is in the $\|\cdot\|_2$-norm induced by the canonical trace $\tau_0$. Moreover, $$u_n^*=\sum_{t\in\Gamma}\overline{u_n(t)}\lambda(t^{-1})$$ 
Let $f \in C(X)$ and $\xi \in l^2 (\Gamma)$ be given. Choose $\epsilon >0$. Let $M=\max\{\sup_{t\in\Gamma}|\xi(t)|,1\}$. Choose $\epsilon'$ such that $2 \epsilon' M\|f\|_\infty < \epsilon /2$.  Let $F$ be a finite subset of $\Gamma$ such that
\[\sum_{t\not\in F}|\xi(t)|^2<(\epsilon')^2.\]
Choose also $\epsilon''>0$ such that $$(\epsilon'' 2\|f\|_\infty + \epsilon'') M|F| < \frac{\epsilon}{2}.$$ Let $A=\{s\in\Gamma: s^{-1}t\mu(f)-f(x)|<\epsilon ~\forall t\in F\}$. Since $u_n$ is $\mu$-contracting towards $x\in X$, we can find a $n_0\in\mathbb{N}$ such that for all $n\ge n_0$, we have
\[\left\|u_n|_A\right\|_2>1-\epsilon''.\]
We now see that 
\begin{align*}
&\left\|\sum_{t\in\Gamma}\left(u_n P_\mu (f) u_n^* \delta_t - f(x)\right) \xi(t)\delta_t\right\|_2\\&\numrel{\le}{rel1} \left\|\sum_{t\in F}\left(u_n P_\mu (f) u_n^* \delta_t - f(x)\right) \xi(t)\delta_t\right\|_2+\left\|\sum_{t\not\in F}\left(u_n P_\mu (f) u_n^* \delta_t - f(x)\right) \xi(t)\delta_t\right\|_2     
\end{align*}
Let us observe that
\[\left\|\sum_{t\not\in F}\left(u_n P_\mu (f) u_n^* \delta_t - f(x)\right) \xi(t)\delta_t\right\|_2\numrel{\le}{rel2} 2\|f\|_{\infty} \sqrt{\sum_{t\not\in F}|\xi(t)|^2}.\]
On the other hand,
\begin{align*}
&\left\|\sum_{t\in F}\left(u_n P_\mu (f) u_n^* - f(x)\right) \xi(t)\delta_t\right\|_2\\&\le \sum_{t\in F}\left\|\left(u_n P_\mu (f) u_n^* - f(x)\right) \xi(t)\delta_t\right\|_2 \\&=\sum_{t \in F} \|(P_\mu (f) u_n^*  - f(x) u_n^*) \xi(t) \delta_t\|_2\\&\le \sum_{t\in F}\left(\left\| \sum_{s \notin A} \overline{u_n(s)} (P_\mu (f) - f(x)) \lambda(s^{-1})\xi(t) \delta_t\right\|_2 +  \left\|\sum_{s \in A}  \overline{u_n(s)} (P_\mu (f) - f(x)) \lambda(s^{-1}) \xi(t)\delta_t\right\|_2\right)\\&=\sum_{t\in F}\left(\left\| \sum_{s \notin A} \overline{u_n(s)} (P_\mu (f) - f(x))\xi(t) \delta_{s^{-1}t}\right\|_2 +  \left\|\sum_{s \in A}  \overline{u_n(s)} (P_\mu (f) - f(x)) \xi(t)\delta_{s^{-1}t}\right\|_2\right)\\&=\sum_{t\in F}\left(\left\| \sum_{s \notin A} \overline{u_n(s)} (s^{-1}t\mu (f) - f(x))\xi(t) \delta_{s^{-1}t}\right\|_2 +  \left\|\sum_{s \in A}  \overline{u_n(s)} (s^{-1}t\mu (f) - f(x)) \xi(t)\delta_{s^{-1}t}\right\|_2\right)\\&\numrel{\le}{rel3} \sum_{t\in F}\left(2\|f\|_{\infty}\sqrt{ \sum_{tu^{-1} \notin A} \left|\overline{u_n(tu^{-1})}\right|^2 |\xi(t)|^2} +  \sqrt{\sum_{tu^{-1} \in A}\left| \overline{u_n(tu^{-1})}\right|^2 \left|(u\mu (f) - f(x))\right|^2 \left|\xi(t)\right|^2}\right)
\end{align*}
Since for all $t\in F$, $|(u\mu (f) - f(x))|<\epsilon$ for every $u$ with $tu^{-1}\in A$ and $\sum_{s\not\in A}|u_n(s)|^2<\epsilon''$, the inequality~(\ref{rel3}) becomes less than or equal to 
\begin{align*}
&\sum_{t\in F}\left( 2 \|f\|_\infty M\epsilon'' + \epsilon\sqrt{\sum_{tu^{-1} \in A}\left| \overline{u_n(tu^{-1})}\right|^2\left|\xi(t)\right|^2}\right)\\
    &\le \sum_{t\in F}\left( 2 \|f\|_\infty M\epsilon'' + M\epsilon\sqrt{\sum_{tu^{-1} \in A}\left| \overline{u_n(tu^{-1})}\right|^2}\right)\\&\numrel{\leq}{rel4} |F|M\left( 2 \|f\|_\infty \epsilon'' + \epsilon''\right).
    \end{align*}
Hence for every $n\ge n_0$, combining the inequalities (\ref{rel1}), (\ref{rel2}) and (\ref{rel4}), we obtain that
\begin{equation*}
        \|u_n P_\mu (f) u_n^* \xi - f(x) \xi\|_2 \leq 2\|f\|_{\infty}\sqrt{\sum_{t\not\in F}|\xi(t)|^2} + |F|M (2 \|f\|_\infty\epsilon'' + \epsilon'') <\epsilon.
    \end{equation*}
The claim follows.    
\end{proof}
\end{lemma}
\section{Crossed product of tracial von Neumann algebras}  \label{tvna} 
In this section, we apply the results of Section \ref{sec:prelim} in order to study the position of the relative commutants of certain subalgebras of crossed-product von Neumann algebras.
\begin{theorem}
\thlabel{tracialvna}
Let $\Gamma$ be a discrete countable group acting on a compact Hausdorff space $X$. Let $\mu$ be a probability measure on $X$ and let $(u_n)$ be a $\mu$-contracting sequence of unitaries in $L\Gamma$.
Let $(\mathcal{N},\tilde{\tau})$ be a tracial von Neumann algebra, and $\Gamma\curvearrowright(\mathcal{N},\tilde{\tau})$ be a trace-preserving action. 
 Then $\{u_n\} ' \cap (\mathcal{N}\rtimes\Gamma) \subset \mathcal{N}\rtimes\Gamma_x$,  where $x$ is determined by the fact that $(u_n)$ is $\mu$-contracting towards $x$.
\end{theorem}
We shall view $\mathcal{N}\rtimes\Gamma\subset \mathbb{B}(L^2(\mathcal{N},\tilde{\tau})\overline{\otimes} \ell^2\Gamma)$. Moreover, 
Let $X$ be a $\Gamma$-space. Let $P_{\mu}:C(X)\to\ell^{\infty}(\Gamma)$ be the Poisson transformation. This gives us a u.c.p $\Gamma$-equivariant map from $C(X)$ to $\ell^{\infty}(\Gamma)$. We can view $\ell^{\infty}(\Gamma)$ as multiplication operators on $\mathbb{B}(\ell^2(\Gamma))$. For $f\in \ell^{\infty}(\Gamma)$, the map $M(f):\ell^2(\Gamma)\to\ell^2(\Gamma)$ defined by $M(f)(\delta_t)=f(t)\delta_t$ is linear and bounded. Therefore, we obtain  a u.c.p map $M\circ P_{\mu}:C(X)\to1\otimes \mathbb{B}(\ell^2(\Gamma)$. We see that every element of $M\circ P_{\mu}(C(X))$ commutes with $\mathcal{N}$. We will ignore $M$ for the most part and write $P_{\mu}(f)$ for the ease of notation.
We denote by $\tau$, the canonical trace $\tilde{\tau}\circ\mathbb{E}$. Note that here $\mathbb{E}:\mathcal{N}\rtimes\Gamma\to\mathcal{N}$ is the canonical conditional expectation. Moreover, we consider the $\|.\|_2$-norm induced by $\tau$. We also denote by $\mathbb{E}_x$, the canonical conditional expectation from $\mathcal{N}\rtimes\Gamma$ to $\mathcal{N}\rtimes\Gamma_x$. 

A crucial ingredient in these arguments is that the state obtained in the limiting stage satisfies some tracial property. While this is automatic in the case of amenable tracial von Neumann algebras, we cannot expect it to hold for us. Nonetheless, before we head on to the proof, we show the existence of an \say{almost-hypertrace}, the last technical bit.
\begin{prop}
\thlabel{almosthyperstate}  
Let $\Gamma$ be a discrete countable group acting on a compact Hausdorff space $X$. Let $\mu$ be a proability measure on $X$ and $(u_n) \subset L\Gamma$ a $\mu$-contracting sequence for the action on $X$. Let $(\mathcal{N},\tilde{\tau})$ be a tracial von Neumann algebra, and $\Gamma\curvearrowright(\mathcal{N},\tilde{\tau})$ be a trace-preserving action. Then, there exists a state $\psi\in S\left(\mathbb{B}(L^2(\mathcal{N},\tilde{\tau})\overline{\otimes} \ell^2\Gamma)\right)$ such that $\psi|_{\mathcal{N}\rtimes\Gamma}=\tau$, $\psi|_{P_{\mu}(C(X))}=\delta_x$, and $\psi\left(aP_{\mu}(f)b\right)=\psi\left(P_{\mu}(f)ab\right)$ for all $a\in \{u_n\}'$ and for all $b\in \mathbb{B}(L^2(\mathcal{N},\tilde{\tau})\overline{\otimes} \ell^2\Gamma)$. In particular, $P_{\mu}(C(X))$ falls in the multiplicative domain of $\psi$.
\begin{proof}
Let $\tau$ denote the canonical trace $\tilde{\tau}\circ\mathbb{E}$. Let $P_\mu : C(X) \rightarrow \mathbb{B}(l^2 \Gamma)$ be the Poisson map, which is ucp $\Gamma$-equivariant. Let $u_n\in L(\Gamma)$ be a $\mu$-contracting sequence towards $x$. We identify an operator $T$ on $\mathbb{B}(\ell^2(\Gamma))$ with $\text{id}\otimes T$, which is an operator on $\mathbb{B}(H\overline{\otimes}\ell^2(\Gamma))$. In this way, we view $\mathbb{B}(\ell^2(\Gamma))$ as a $\Gamma$-invariant subalgebra of $\mathbb{B}(H\overline{\otimes}\ell^2(\Gamma))$. Moreover, if $T_n\in \mathbb{B}(\ell^2\Gamma)$ is a uniformly bounded sequence such that $T_n\xrightarrow[\mathbb{B}(\ell^2\Gamma)]{\text{SOT}} T$, then, 
$\text{id}\otimes T_n\xrightarrow[\mathbb{B}(H\overline{\otimes}\ell^2\Gamma)]{\text{SOT}} \text{id}\otimes T$. Therefore, using \thref{SOTconvergence}, we see that for every $f \in C(X)$, $$u_n(P_\mu (f)) u_n^* \xrightarrow[\text{SOT}]{n\to\infty} f(x) \cdot 1.$$ In particular, since $(u_n(P_\mu (f)) u_n^*)$ is uniformaly bounded, this implies that 
$$u_n(P_\mu (f))(P_\mu (f))^* u_n^* \xrightarrow[\text{SOT}]{n\to\infty} f(x)\overline{f(x)} \cdot 1.$$    
Consider now the (separable) $C^*$-algebra $A$ generated by $P_\mu (C(X))$. Note that the state $\tau$ is of the form $\hat{1}_{\mathcal{N}}\otimes\delta_e$, and hence, is defined on $\mathbb{B}(L^2(\mathcal{N},\tilde{\tau})\overline{\otimes} l^2\Gamma)$. Consider, after passing to a subnet if necessary, a weak$^*$ limit $$\psi(\cdot):= \lim_n \tau \circ ad (u_n)(\cdot)=\lim_n\left\langle (\cdot) \hat{1}_{\mathcal{N}}\otimes\hat{u}_n, \hat{1}_{\mathcal{N}}\otimes\hat{u}_n\right\rangle\in S(\mathbb{B}(L^2(\mathcal{N},\tilde{\tau})\overline{\otimes} l^2\Gamma)).$$We see that $P_\mu(f)$ is in the multiplicative domain of $\psi$ for every $f \in C(X)$.\\
    \textit{Claim}: $\psi ((aP_\mu (f) - P_\mu (f)a) (aP_\mu (f)-P_\mu (f)a)^*) =0$ for all $a$ which commute with $\{u_n:n\in\mathbb{N}\}$.\\
Let us observe that \begin{align*}&\psi ((aP_\mu (f) -P_\mu (f)a)(aP_\mu (f)-P_\mu (f)a)^*)\\&= \psi(aP_{\mu}(f)P_{\mu}(f)^*a^*)-\psi (P_{\mu}(f)aP_{\mu}(f)^*a^*)\\&-\psi (aP_{\mu}(f)a^*P_{\mu}(f)^*)+\psi (P_{\mu}(f)aa^*P_{\mu}(f)^*).\end{align*} 
Now, since $\phi$ is normal and $a$ commutes with $\{u_n:n\in\mathbb{N}\}$,
\begin{align*}&\psi(aP_{\mu}(f)P_{\mu}(f)^*a^*)\\&=\lim_n \phi( u_na P_{\mu}(f) P_{\mu} (f)^* a^*u_n^*)\\&=\lim_n \phi(a u_n P_{\mu}(f) P_{\mu} (f)^* u_n^* a^*)\\&=|f(x)|^2\phi(aa^*).\end{align*}  On the other hand, since $P_\mu (f)$ and $P_\mu (f)^*$ are in the multiplicative domain, we have $\psi (P_\mu (f) aa^* P_\mu (f)^*) = f(x) \overline{f(x)} \psi (aa^*)$. Since $a$ commutes with $\{u_n:n\in\mathbb{N}\}$, we see that $\psi (aa^*)=\lim_n\phi(u_naa^*u_n^*)=\phi(aa^*)$. Therefore, $$\psi (P_\mu (f) aa^* P_\mu (f)^*) = f(x) \overline{f(x)} \psi (aa^*)=|f(x)|^2\phi(aa^*).$$  Arguing similarly, we see that
\begin{align*}
\psi(P_{\mu}(f)aP_{\mu}(f)^*a^*)&=f(x)\psi(aP_{\mu}(f)^*a^*)\\&=f(x)\lim_n\phi(u_naP_{\mu}(f)^*a^*u_n^*)\\&=f(x)\lim_n\phi(au_nP_{\mu}(f)^*u_n^*a^*)\\&=f(x)\overline{f(x)}\phi(aa^*)\\&=|f(x)|^2 \phi(aa^*).   
\end{align*}
It also follows similarly that 
\[\psi (aP_{\mu}(f)a^*P_{\mu}(f)^*)=|f(x)|^2\phi(aa^*).\]
Consequently, we see that $\psi((aP_\mu (f) - P_\mu (f)a) (aP_\mu (f)-P_\mu (f)a)^*) =0$ for every $a \in \{u_n\}'$.
Moreover, for every $b\in\mathbb{B}(L^2(\mathcal{N},\tilde{\tau})\overline{\otimes} l^2\Gamma)$ and $a\in L\Lambda'$, we see that
\begin{align*}&|\psi(P_\mu (f) a b - a P_\mu (f) b)|^2 \\&\leq \psi ((P_\mu (f) a - aP_\mu (f)) (P_\mu (f) a - a P_\mu (f))^*) \psi (b^* b)\\&=0\end{align*}
Therefore, it follows that $\psi(P_\mu (f)ab)=\psi(a P_\mu (f) b)$ for all $a\in \{u_n\}'$ and $b\in \mathbb{B}(L^2(\mathcal{N},\tilde{\tau})\overline{\otimes} l^2\Gamma)$.   
\end{proof}
\end{prop}

Our idea of the proof is motivated by \cite[Theorem~1.4]{boutonnet2015maximal}.
\begin{proof}[Proof of \thref{tracialvna}] 
From \thref{almosthyperstate}, we can find a state $\psi\in S\left(\mathbb{B}(L^2(\mathcal{N},\tilde{\tau})\overline{\otimes} \ell^2\Gamma)\right)$ such that $\psi|_{\mathcal{N}\rtimes\Gamma}=\tau$, $\psi|_{P_{\mu}(C(X))}=\delta_x$, and $\psi\left(aP_{\mu}(f)b\right)=\psi\left(P_{\mu}(f)ab\right)$ for all $a\in \{u_n\}'$ and for all $b\in \mathbb{B}(L^2(\mathcal{N},\tilde{\tau})\overline{\otimes} \ell^2\Gamma)$. Let $\mathcal{M}$ denote the von Neumann algebra  $\{u_n\} ' \cap (\mathcal{N}\rtimes\Gamma)$. Let $u\in\mathcal{M}$ be a unitary element. We shall show that $\|\mathbb{E}_{x}(u)\|_2=1$ from whence it will follow that $u\in \mathcal{N}\rtimes\Gamma_x$. Let $\epsilon>0$. Let $u_0=\sum_{i=1}^na_i\lambda(s_i)\in\mathcal{N}\rtimes\Gamma$ be such that \begin{equation}
\label{firstapprox}
\|u^*-u_0\|_2<\epsilon.\end{equation} Let us write $F=\{s_1,s_2,\ldots,s_n\}$. Then, we can rewrite $$u_0=\sum_{s\in F\cap\Gamma_x}a_s\lambda(s)+\sum_{s\in F\cap\Gamma_x^c}a_s\lambda(s)$$ In particular, we see that $sx\ne x$ for all $s\in F\cap\Gamma_x^c$. Therefore, we can find $f\in C(X)$ with $0\le f\le 1$ such that $f(x)=1$ and $f(sx)=0$ for all $s\in F\cap\Gamma_x^c$. Below, we write $f$ instead of $P_{\mu}(f)$ for ease of notation.
Let us now observe that 
\begin{align*}
\left|\psi\left(f(uu_0-1)\right)\right|&\le \sqrt{\psi\left((uu_0-1)^*(uu_0-1)\right)}\sqrt{\psi(ff^*)}\\&=\|uu_0-1\|_2&\text{$\left(\psi|_{\mathcal{N}\rtimes\Gamma}=\tau\right)$}\\&\le \|u^*-u_0\|_2<\epsilon.    
\end{align*}
Therefore,
\begin{align*}
\left|\psi\left(ufu_0\right)\right|&= \left|\psi\left((fuu_0\right)\right|\\&=\left|\psi\left(f(uu_0-1)\right)+\psi\left(f\right)\right|\\&\ge \left|\psi(f)-|\psi\left(f(uu_0-1)\right)|\right|\ge 1-\epsilon.    
\end{align*}
To reiterate, 
\begin{equation}
\label{ineq:firstinequality}
\left|\psi\left(ufu_0\right)\right|\ge 1-\epsilon.   \end{equation}
On the other hand,
\begin{align*}
&\left|\psi\left((ufu_0\right)\right|\\&\le\left|\psi\left(uf\left(\sum_{s\in F\cap\Gamma_x}a_s\lambda(s)\right)\right)\right|+  \left|\psi\left(uf\left(\sum_{s\in F\cap\Gamma_x^c}a_s\lambda(s)\right)\right)\right|\\&\le \left|\psi\left(uf\mathbb{E}_{x}(u_0)\right)\right|+\sum_{s\in F\cap\Gamma_x^c}\left|\psi\left(ufa_s\lambda(s)\right)\right| 
\end{align*}
Since $f\in C(X)$, $a_s\in \mathcal{N}$ and every element of $C(X)$ commutes with $\mathcal{N}$ (see the paragraph below \thref{tracialvna}), we see that 
\begin{align*}\sum_{s\in F\cap\Gamma_x^c}\left|\psi\left(ufa_s\lambda(s)\right)\right|&=\sum_{s\in F\cap\Gamma_x^c}\left|\psi\left(ua_sf\lambda(s)\right)\right|\\&= \sum_{s\in F\cap\Gamma_x^c}\left|\psi\left(ua_s\lambda(s)s^{-1}f\right)\right|\\&=\sum_{s\in F\cap\Gamma_x^c}\left|\psi\left(ua_s\lambda(s)\right)f(sx)\right|\\&=0\end{align*}
It follows therefore that $\left|\psi\left((ufu_0\right)\right|\le \left|\psi\left(uf\mathbb{E}_{x}(u_0)\right)\right|$.
Combining this along with equation~\eqref{firstapprox}, equation~\eqref{ineq:firstinequality} and the Cauchy-Schwartz inequality, we see that
\begin{align*}
1-\epsilon\le \left|\psi(ufu_0)\right|&\le\left|\psi\left(uf\mathbb{E}_{x}(u_0)\right)\right|\\&\le\sqrt{\psi(uff^*u^*)}\left\|\mathbb{E}_{x}(u_0)\right\|_2\\&\le\left\|\mathbb{E}_{x}(u)\right\|_2+\epsilon
\end{align*}
As a result, it follows that $\left\|\mathbb{E}_{x}(u)\right\|_2\ge 1-2\epsilon$. Since $\epsilon>0$ is arbitrary, it follows that $u\in\mathcal{N}\rtimes\Gamma_x$.
\end{proof}
We obtain the following as an immediate result.
\begin{cor}
\thlabel{relativecommutantforgroupvna}
Let $\Gamma$ be a discrete countable group acting on a compact Hausdorff space $X$. Let $(u_n)$ be a $\mu$-contracting sequence for the action on $X$ for some probability measure $\mu$ on $X$. Then $\{u_n\}' \cap L\Gamma \subset L(\Gamma_x)$,  where $x$ is determined by the fact that $u_n$ is  $\mu$-contracting towards $x$.
\end{cor}
\begin{example}
Let $\Gamma$ be a convergence group. Then, in view of Example \ref{ex:convggrp}, given a parabolic element $s \in \Gamma$, for every tracial crossed product $\mathcal{N} \rtimes \Gamma$ and every diffuse subalgebra $\mathcal{M}$ of $L(\langle s \rangle)$, the relative commutant of $\mathcal{M}$ in $    \mathcal{N} \rtimes \Gamma$ is injective. If $s \in \Gamma$ is loxodromic, then $L(\langle s \rangle)' \cap (\mathcal{N} \rtimes \Gamma)$ is injective.
    \end{example}
\section{Relative commutants of subgroups of negatively curved groups}\label{ncv}
In this section, we examine the relative commutants of subgroups inside groups that satisfy \say{north pole-south pole}-dynamics. We begin with the following singularity phenomenon, which has been exploited in the past to prove rigidity results (see, for example,  \cite{kalantar_kennedy_boundaries, HartKal, KalPan, BBHP, AH24} etc). 
\begin{lemma}
\thlabel{twodelta}
Let $X$ be a continuous $\Gamma$-space. Let $\tau\in S\left(C(X)\rtimes_r\Gamma\right)$ such that $\tau|_{C(X)}=a\delta_x+(1-a)\delta_y$ for some $x\ne y\in X$. Then, 
$\tau(\lambda(s))=0$ for all $s \in \Gamma$ with $s\{x,y\}\cap\{x,y\}=\emptyset$.
\begin{proof}
Let $s\in \Gamma$ be such that $s\{x,y\}\cap\{x,y\}=\emptyset$. Using Uryhson's lemma, we can find a non-negative continuous function $f\in C(X)$ with $0<f<1$ such that $f|_{\{x,y\}}=1$ and $f|_{\{sx,sy\}}=0$. Using Cauchy-Schwartz inequality, we obtain
\begin{align*}
\left|\tau(f\lambda(s))\right|^2&=\left|\tau\left(\sqrt{f}\sqrt{f}\lambda(s)\right)\right|^2\\&\le\tau\left(f\right)\tau\left(\lambda(s^{-1})f\lambda(s)\right)\\&=\tau(f)\tau(s^{-1}.f)
\end{align*}
Since $\tau|_{C(X)}=a\delta_x+(1-a)\delta_y$, we obtain that \[\tau(s^{-1}f)=as^{-1}.f(x)+(1-a)s^{-1}.f(y)=af(sx)+(1-a)f(sy)=0.\]
This shows that $\tau(f\lambda(s))=0$.
On the other hand, applying Cauchy-Schwartz inequality again,
\begin{align*}
|\tau\left((1-f)\lambda(s)\right)|^2&=\left|\tau\left(\sqrt{1-f}\sqrt{1-f}\lambda(s)\right)\right|^2\\&\le\tau\left(1-f\right)\tau\left(\lambda(s^{-1})(1-f)\lambda(s)\right)\\&=\tau(1-f)\tau(s^{-1}.(1-f))
\end{align*}
Let us now see that 
\[\tau(1-f)=a(1-f(x))+(1-a)(1-f(y))=a(0)+(1-a)(0)=0\]
Therefore, we obtain that $\tau\left((1-f)\lambda(s)\right)=0$.
Now, combining the above two identities, we see that
\[\tau(\lambda(s))=\tau(f\lambda(s))+\tau\left((1-f)\lambda(s)\right)=0.\]
This completes the proof.
\end{proof}
\end{lemma}
An action $\Gamma\curvearrowright X$ is said to have \say{north pole-south pole}-dynamics, if for every infinite order element $g\in\Gamma$, there are unique fixed point $x_g^+$ and $x_g^-$ on the $\Gamma$-space $X$ such that $g^nx\xrightarrow[]{n\to\infty}x_g^+$ for all $x\ne x_g^{-}$. We denote by $E(g)=\text{Stab}_{\Gamma}(\{x_g^+, x_g^-\})$, the set wise stabilizer of  $\{x_g^+, x_g^-\}$. 
We denote by $\mathbb{E}_{E(g)}$ the canonical conditional expectation from $C_r^*(\Gamma)$ onto $C_r^*(E(g))$. This also extends to a normal trace-preserving conditional expectation from $L(\Gamma)$ onto $L(E(g))$.
\begin{prop}
\thlabel{plclosure}
Let $\Gamma$ be a discrete group admitting a minimal action $\Gamma\curvearrowright X$ with the north pole-south pole-dynamics. Let $s\in\Gamma$ be an infinite order element with the property that $t\{x_s^+,x_s^-\}\cap\{x_s^+,x_s^-\}=\emptyset$ for all $t\not\in E(s)$. Then, given $a\in C_r^*(\Gamma)$ and $\epsilon>0$, we can find $\{s_1,s_2,\ldots,s_m\}\subset\langle s \rangle$ such that
\[\left\|\frac{1}{m}\sum_{j=1}^m\lambda(s_j)\left(a-\mathbb{E}_{E(s)}(a)\right)\lambda(s_j)^*\right\|<\epsilon.\]
\end{prop}
Before we head on to the proof, let us briefly ponder our strategy, similar to that of \cite{haagerup2016new}. Let $\Lambda=\langle s \rangle$. We shall first show that for every bounded linear functional $\varphi$ on  $S(C_r^*(\Gamma))$, we can find a bounded linear functional $\psi\in \overline{\{s.\omega: s\in\Lambda\}}^{\text{weak}^*}$ such that $\psi=\psi\circ\mathbb{E}_{\Lambda}$. Here, $\mathbb{E}_{\Lambda}: C_r^*(\Gamma)\to C_r^*(\langle s \rangle)$ is the canonical conditional expectation. The claim would then follow by a usual Hahn-Banach separation argument. 
\begin{proof}
Let $\Gamma$ be a discrete group admitting a minimal action $\Gamma\curvearrowright X$ with the north pole-south pole dynamics. Since $\Gamma\curvearrowright X$ is minimal, we can view $C(X)$ as multiplication operators on $\mathbb{B}(\ell^2(\Gamma))$. 
Given a bounded linear functional $\varphi$ on $C_r^*(\Gamma)$, extend it to a bounded linear functional $\eta$ on $C(X)\rtimes_r\Gamma$. We can write $\eta=c_1\omega_1-c_2\omega_2+ic_3\omega_3-ic_4\omega_4$, where $\omega_i\in S(C(X)\rtimes_r\Gamma)$ and $c_i\in\mathbb{C}$ for each $i=1,2,3,4$. Let $\nu_i=\omega_i|_{C(X)}$ for each $i=1,2,3,4$. Since $s$ is an infinite order element, there are unique fixed points $x_s^+$ and $x_s^-$ on the $\Gamma$-space $X$ such that $s^nx\xrightarrow[]{n\to\infty}x_s^+$ for all $x\ne x_s^{-}$. Using the dominated convergence theorem, it follows that $s^n\nu_i\xrightarrow[]{\text{weak}^*}a_i\delta_{x_s^+}+(1-a_i)\delta_{x_s^{-}}$, where $a_i=\nu_i(X\setminus x_s^{-})$. By passing to a subnet (four times) if required, we can assume that $s^n\omega_i\to\omega_i'\in S(C(X)\rtimes_r\Gamma)$ for each $i=1,2,3,4$. Observe that $\omega_i'|_{C(X)}=a_i\delta_{x_s^+}+(1-a_i)\delta_{x_s^{-}}$. Let $\eta'=c_1\omega_1'-c_2\omega_2'+ic_3\omega_3'-ic_4\omega_4'$. 
Let $\psi=\eta'|_{C_r^*(\Gamma)}$. We claim that $\psi=\psi\circ\mathbb{E}_{E(s)}$. Note that $t\{x_s^+,x_s^-\}\cap \{x_s^+,x_s^-\}=\emptyset$ for all $t\not\in E(s)$. It now follows from \thref{twodelta} that $\omega_i'(\lambda(t))=0$ for all $t\not\in E(s)$ and for all $i=1,2,3,4$.
This shows that $\omega_i'|_{C_r^*(\Gamma)}=\omega_i'\circ\mathbb{E}_{E(s)}$ for each $i=1,2,3,4$. Consequently, it follows that $\psi=\psi\circ\mathbb{E}_{E(s)}$. The claim now follows by a usual Hahn-Banach separation argument (see, for example, \cite[Theorem~3.4]{bryder2018reduced}).
\end{proof}
It was shown in \cite{amrutam2020simplicity} that the averaging scheme at the level of the group $C^*$-algebra lifts to the same averaging scheme at the level of the crossed product. We merely reiterate the steps to prove that the averaging established in \thref{plclosure} lifts to the crossed product of tracial von Neumann algebras. Given a tracial von Neumann algebra $(\mathcal{N},\Tilde{\tau})$ and a trace preserving action $\Gamma\curvearrowright (\mathcal{N},\Tilde{\tau})$, we let $\tau=\tilde{\tau}\circ\mathbb{E}$ which is a faithful normal trace on $\mathcal{N}\rtimes\Gamma$. We do all the approximations in the $\|\cdot\|_2$-norm induced by $\tau$. We denote by $\mathbb{E}$, the canonical conditional expectation from $\mathcal{N}\rtimes\Gamma$ onto $\mathcal{N}$. Moreover, we shall use $\Tilde{\mathbb{E}}_{E(s)}$ to denote the canonical conditional expectation from $\mathcal{N}\rtimes\Gamma$ onto $\mathcal{N}\rtimes E(s)$.
\begin{theorem}
\thlabel{avgtracialvna}
Let $\Gamma$ be a discrete group admitting a minimal action $\Gamma\curvearrowright X$ with the north pole-south pole-dynamics. Let $s\in\Gamma$ be an infinite order element with the property that $t\{x_s^+,x_s^-\}\cap\{x_s^+,x_s^-\}=\emptyset$ for all $t\not\in E(s)$. 
Let $(\mathcal{N},\Tilde{\tau})$ be a tracial von Neumann algebra and $\Gamma\curvearrowright(\mathcal{N},\Tilde{\tau})$ be a trace-preserving action. Let $\mathcal{M}=\mathcal{N}\rtimes\Gamma$. Then, given $x\in\mathcal{M}$ and $\epsilon>0$, we can find $\{s_1,s_2,\ldots,s_m\}\subset\langle s \rangle$ such that
\[\left\|\frac{1}{m}\sum_{j=1}^m\lambda(s_j)\left(x-\Tilde{\mathbb{E}}_{E(s)}(x)\right)\lambda(s_j)^*\right\|_2<\epsilon.\]
\begin{proof}
Let $x\in \mathcal{M}$ and $\epsilon>0$ be given. We can find a finite set $F\subset\Gamma$ and finitely many elements $\{a_t: t\in F\}\subset\mathcal{N}$ such that 
\[\left\|x-\sum_{t\in F}a_t\lambda(t)\right\|_2<\frac{\epsilon}{3}.\]
Since $\mathbb{E}\circ\Tilde{\mathbb{E}}_{E(s)}=\mathbb{E}$, it follows that
\[\left\|\Tilde{\mathbb{E}}_{E(s)}(x)-\sum_{t\in E(s)\cap F}a_t\lambda(t)\right\|_2<\frac{\epsilon}{3}.\]
Let $M=\sup_{t\in F}\|a_t\|$. Using \thref{plclosure} for $a=\sum_{t\in E(s)^c\cap F}\lambda(t)$, we can find $\{s_1,s_2,\ldots,s_m\}\subset \langle s \rangle$ such that 
\[\left\|\frac{1}{m}\sum_{j=1}^m\lambda(s_j)\left(\sum_{t\in E(s)^c\cap F}\lambda(t)\right)\lambda(s_j)^*\right\|<\frac{\epsilon}{3|F|M}.\]
It follows from \cite[Lemma~4.1]{haagerup2016new} that
\[\left\|\frac{1}{m}\sum_{j=1}^m\lambda(s_j)\lambda(t)\lambda(s_j)^*\right\|_2\le \left\|\frac{1}{m}\sum_{j=1}^m\lambda(s_j)\lambda(t)\lambda(s_j)^*\right\|<\frac{\epsilon}{3|F|M},~\forall t\in E(s)^c\cap F.\]
Therefore, using \cite[Lemma~2.1]{amrutam2020simplicity}, we obtain that
\begin{align*}\left\|\frac{1}{m}\sum_{j=1}^m\lambda(s_j)\left(\sum_{t\in E(s)^c\cap F}a_t\lambda(t)\right)\lambda(s_j)^*\right\|&\le \sum_{t\in E(s)^c\cap F}\|a_t\|\left\|\frac{1}{m}\sum_{j=1}^m\lambda(s_j)\lambda(t)\lambda(s_j)^*\right\|\\&\le\sum_{t\in E(s)^c\cap F}\|a_t\|\frac{\epsilon}{3|F|M}<\frac{\epsilon}{3}.\end{align*}
Putting all these together, along with an application of triangle inequality, we see that
\begin{align*}
&\left\|\frac{1}{m}\sum_{j=1}^m\lambda(s_j)\left(x-\Tilde{\mathbb{E}}_{E(s)}(x)\right)\lambda(s_j)^*\right\|_2\\&\le\left\|\frac{1}{m}\sum_{j=1}^m\lambda(s_j)\left(x-\sum_{t\in F}a_t\lambda(t)\right)\lambda(s_j)^*\right\|_2\\&+\left\|\frac{1}{m}\sum_{j=1}^m\lambda(s_j)\left(\sum_{t\in F\cap E(s)}a_t\lambda(t)-\Tilde{\mathbb{E}}_{E(s)}(x)\right)\lambda(s_j)^*\right\|_2\\&+\left\|\frac{1}{m}\sum_{j=1}^m\lambda(s_j)\left(\sum_{t\in E(s)^c\cap F}a_t\lambda(t)\right)\lambda(s_j)^*\right\|_2\\&\le \left\|x-\sum_{t\in F}a_t\lambda(t)\right\|_2+\left\|
\sum_{t\in F\cap E(s)}a_t\lambda(t)-\Tilde{\mathbb{E}}_{E(s)}(x)\right\|_2+\left\|\frac{1}{m}\sum_{j=1}^m\lambda(s_j)\left(\sum_{t\in E(s)^c\cap F}a_t\lambda(t)\right)\lambda(s_j)^*\right\|\\&\le \frac{\epsilon}{3}+\frac{\epsilon}{3}+\frac{\epsilon}{3}=\epsilon.
\end{align*}
The claim follows.
\end{proof}
\end{theorem}
Consequently, we can determine the position of the relative commutants of specific subgroups of a non-elementary acylindrically-hyperbolic group with at least one infinite order element. We briefly recall the definitions and refer the readers to \cite{OS} for more details.
\subsection*{Acylindrically Hyperbolic groups} An action $\Gamma\curvearrowright (X,d)$ on a metrizable space is considered acylindrical if for every $\epsilon > 0$, there exist $\delta, N > 0$ such that for any $x, y \in X$ with $d(x, y) \geq \delta$, the number of elements $g \in \Gamma$ satisfying $d(x, gx) \leq \epsilon$ and $d(y, gy) \leq \epsilon$ is at most $N$. A group $\Gamma$ is called acylindrically hyperbolic if it admits a non-elementary acylindrical action on a hyperbolic space. 

Every non-elementary hyperbolic group is acylindrically hyperbolic. Further examples of acylindrically hyperbolic groups include non-(virtually) cyclic groups hyperbolic relative to proper subgroups, $\text{Out}(F_n)$ for $n > 1$, many mapping class groups, and non-(virtually cyclic) groups acting properly on proper CAT($0$)-spaces and containing rank one elements, among others (for more details, refer to \cite[Section~8]{OS} and the references therein).

For a group $\Gamma$ acting on a hyperbolic space $S$, recall that an infinite order element $g\in \Gamma$ is called loxodromic if it has precisely two fixed points $x_g^+,x_g^-$ on the Gromov boundary $\partial S$ and $g^nx\to x_g^+$ for every $x\in \partial S$ except $x_g^-$. It turns out that a group being acylindrically hyperbolic is equivalent to the notion of \say{weak proper discontinuity} introduced by Bestvina and Fujiwara \cite{bestvina2002bounded}. Let $\Gamma$ be a group acting on a hyperbolic space $S$. An element $g\in\Gamma$ is said to have the weak proper discontinuity property (in this case, we say that $g$ is a WPD element) if for every $\epsilon > 0$ and every $x\in S$, there exist $M \in\mathbb{N}$ such that the number of elements $h \in \Gamma$ satisfying $d(x, hx) <\epsilon$ and $d(g^Mx, hg^Mx) < \epsilon$ is finite. 

Osin \cite[Theorem~1.2]{OS} later established that a group $\Gamma$ being acylindrically hyperbolic is equivalent to the existence of a loxodromic element $g\in\Gamma$ that satisfies the weak proper discontinuity condition. Moreover, there is a unique maximal virtually cyclic subgroup $E(g)\le \Gamma$ containing $g$. Explicitly, $E(g)=\text{Stab}_{\Gamma}(\{x_g^+, x_g^-\})$ is the set wise stablizer of  $\{x_g^+, x_g^-\}$ (see for example, \cite[Lemma~6.5]{dahmani2017hyperbolically}).

\begin{cor}
\thlabel{relcommutantinhyp}
Let $\Gamma$ be a group admitting an action $\Gamma\curvearrowright X$ with north pole-south pole-dynamics. Let $\Lambda\le \Gamma$ be a subgroup with one infinite order element $s\in\Lambda$. Assume that $\{tx_s^+, tx_s^-\}\cap \{x_s^+,x_s^-\}=\emptyset$ for all $t\not\in \text{Stab}(\{x_s^+,x_s^-\})$. Let $(\mathcal{N},\Tilde{\tau})$ be a tracial von Neumann algebra and $\Gamma\curvearrowright(\mathcal{N},\Tilde{\tau})$ be a trace-preserving action. Let $\mathcal{M}=\mathcal{N}\rtimes\Gamma$. Then, $L(\Lambda)'\cap \mathcal{M}\subset \mathcal{N}\rtimes E(s)$.
\begin{proof}
Since $s$ is an infinite order WPD loxodromic element, it satisfies the north pole-south pole dynamics on the Gromov boundary.
Let $x\in L(\Lambda)'\cap \mathcal{M}$. Let $s\in\Lambda$ be an infinite order loxodromic element. Let $\epsilon>0$. Using \thref{avgtracialvna}, it follows that we can find $s_1,s_2,\ldots,s_m\subset\langle s \rangle$ such that
\begin{equation}
\label{eq:2norm}
\left\|\frac{1}{m}\sum_{j=1}^m\lambda(s_j)\left(x-\Tilde{\mathbb{E}}_{E(s)}(x)\right)\lambda(s_j)^*\right\|_2<\epsilon.\end{equation}
Note that here $\Tilde{\mathbb{E}}_{E(s)}:\mathcal{N}\rtimes\Gamma\to\mathcal{N}\rtimes E(s)$ is the canonical conditional expectation.
Let us write $\Tilde{\mathbb{E}}_{E(s)}(x)=\sum_{t\in E(s)}a_t\lambda(t)$, where the convergence is in the $\|\cdot\|_2$-norm. For any $s_j\in \langle s\rangle$, writing it as $s^{m_j}$ for some $m_j\in\mathbb{Z}$, we see that
\[\lambda(s_j)\Tilde{\mathbb{E}}_{E(s)}(x)\lambda(s_j)^*=\sum_{t\in E(s)}\alpha_{s_j}(a_t)\lambda(s_jts_j^{-1})=\sum_{t\in E(s)}\alpha_{s_j}(a_t)\lambda(s^{m_j}ts^{-m_j}).\]
Since $t\in E(s)=\{x_s^+, x_s^-\}$, we see that $s^{m_j}ts^{-m_j}\{x_s^+,x_s^-\}=\{x_s^+,x_s^-\}$. As such, we can now see that  $\lambda(s_j)\Tilde{\mathbb{E}}_{E(s)}(x)\lambda(s_j)^*\in \mathcal{N}\rtimes E(s)$ for each $j=1,2,\ldots,m$. Writing $$\sum_{j=1}^m\lambda(s_j)\Tilde{\mathbb{E}}_{E(s)}(x)\lambda(s_j)^*=y_{E(s)},$$ since $x\in L(\Lambda)'\cap \mathcal{M}$, it follows from equation~\eqref{eq:2norm} that
\[\left\|x-y_{E(s)}\right\|_2<\epsilon.\]
Since $\epsilon>0$ is arbitrary, it is evident that $x\in \mathcal{N}\rtimes E(s)$. Therefore, $L(\Lambda)'\cap \mathcal{M}\subset \mathcal{N}\rtimes E(s)$.
\end{proof}
\end{cor}
In addition, if we assume that $\mathcal{N}$ is amenable, then it can be concluded that the relative commutant $L(\Lambda)'\cap \mathcal{M}$ is amenable. 

We now give examples that fit into the above setup. Before doing so, we briefly recall the notion of hyperbolic elements and refer the reader to \cite{ghys2013groupes} for more details. Let $\Gamma$ be a group acting by isometries on a hyperbolic space $X$. An element $s\in\Gamma$ is called hyperbolic if it fixes exactly two points on the boundary of $X$, denoted by $\partial\Gamma$.
\begin{example}
Let $\Gamma$ be a torsion-free hyperbolic group and $\Lambda\le \Gamma$ be an infinite subgroup. Then, $\Lambda$ contains an element of infinite order and is loxodromic. Let's call it $s$. Denote by $x_s^+$ and $x_s^-$, the corresponding fixed points on the Gromov boundary $\partial\Gamma$. It is well-known that $E(s)=\text{Stab}_{\Gamma}\left(\{x_s^+, x_s^-\}\right)$ (see for example \cite[Lemma~6.5]{dahmani2017hyperbolically}).
Since $s\in \text{Fix}(x_s^+)$ is a hyperbolic element, and $x_s^-$ is the unique fixed point of $x$ in $\partial\Gamma\setminus x_s^+$, it follows from the proof of \cite[Theorem~8.30]{ghys2013groupes} that $\text{Fix}(x_s^+)=\text{Fix}(x_s^-)$.
We claim that $\{tx_s^+, tx_s^-\}\cap \{x_s^+,x_s^-\}=\emptyset$ for all $t\not\in E(s)$. Let $t\not\in E(s)$. Since $\text{Fix}(x_s^+)=\text{Fix}(x_s^-)$, it follows that $tx_s^+\ne x_s^+$ and $tx_s^-\ne x_s^-$. If $tx_s^+=x_s^-$, then $t^{-1}stx_s^+=x_s^+$. Therefore, $t^{-1}st\in \text{Fix}(x_s^-)$. Therefore, we see that $t^{-1}stx_s^-=x_s^-$. This further implies that $s(tx_s^-)=(tx_s^-)$. Since $s$ is a loxodromic element, either $tx_s^-=x_s^-$ or $tx_s^-=x_s^+$. If $tx_s^-=x_s^-$, it would follow that $tx_s^+=x_s^-=tx_s^-$ from whence we would obtain that $x_s^+=x_s^-$ which would contradict the fact that $s$ is an infinite loxodromic element. Therefore, $tx_s^-=x_s^+$. This shows that $t\in E(s)$ which contradicts our earlier choice of $t\not\in E(s)$. If $tx_s^-=x_s^+$, the argument follows analogously by replacing $t^{-1}st$ with $tst^{-1}$. As such, we can now apply \thref{relcommutantinhyp} to conclude the relative commutant $L(\Lambda)'\cap \mathcal{N}\rtimes\Gamma$ is contained inside $\mathcal{N}\rtimes E(s)$ for any trace-preserving action $\Gamma\curvearrowright (\mathcal{N},\Tilde{\tau})$. Under the further assumption of amenability of $\mathcal{N}$, it follows that $L(\Lambda)'\cap \mathcal{N}\rtimes\Gamma$ is amenable since $E(s)$ is an amenable subgroup of $\Gamma$.
\end{example}

There are many acylindrically hyperbolic groups for which we can find an element $t\not\in E(s)$ such that $t\{x_s^+,x_s^-\}\cap\{x_s^+,x_s^-\} \ne\emptyset$. Nevertheless, we can still determine the position of the relative commutant of any diffuse von Neumann subalgebra of $L(\langle s\rangle) $ in these situations. Recall that for an action $\Gamma\curvearrowright X$ with north pole-south pole-dynamics, an element $s\in\Gamma$ is called parabolic if there exists a unique fixed point $x_s^+\in X$ such that both $\{s^nx\}_{n\in \mathbb{N}}$ and $\{s^{-n}x\}_{n\in \mathbb{N}}$ converge to  $x_s^+$ as $n\to \infty$ for every $x\in X$.
\begin{cor}
\thlabel{diffnorthsouth}
Let $\Gamma$ be a discrete group admitting an action $\Gamma\curvearrowright X$ with the north pole-south pole-dynamics. Let $s\in\Gamma$ be an infinite order parabolic element. Suppose there exists a quasi-invariant probability measure $\mu\in\text{Prob}(X)$ such that $\mu(x_s^+)\ne 0$.
Let $(\mathcal{N},\Tilde{\tau})$ be a tracial von Neumann algebra and $\Gamma\curvearrowright(\mathcal{N},\Tilde{\tau})$ be a trace-preserving action. Let $\mathcal{M}=\mathcal{N}\rtimes\Gamma$. Then, $\mathcal{M}_1'\cap\mathcal{M}\le \mathcal{N}\rtimes\Gamma_{x_s^+}$ for any diffuse subalgebra $\mathcal{M}_1\le L(\langle s\rangle)$.
\begin{proof}
Let $s\in\Gamma$ be an infinite order parabolic element. By assumption, there exists one fixed point $x_s^+$ on $X$. Moreover, $s^ny\xrightarrow{n\to\infty} x_s^+$ for all $y\in X$. Let $P_{\mu}:C(X)\to \mathbb{B}(\ell^2(\Gamma))$ be the associated Poisson-transformation. It follows from the definition~\ref{def:contractionforgroupelements} that $(s^n)$ is $\mu$ contracting towards $x_s^+$. In particular, every diverging sequence $(\lambda_n)\subset \Lambda=\langle s \rangle$ is $\mu$ contracting towards $x_s^+$. Let $\mathcal{M}_1\le L(\langle s\rangle)$ be a diffuse subalgebra. Let $(u_n)\subset\mathcal{U}(\mathcal{M}_1)$ be a sequence of unitaries which converge to $0$ weakly. We can now appeal to Lemma~\ref{lemb} to conclude that $u_n$ is $\mu$-contracting towards $x_s^+$. The claim now follows from \thref{tracialvna}.
\end{proof}
\end{cor}
Recall that a CAT$(0)$-cube complex is a simply connected cell complex whose cells are Euclidean cubes $[0,1]^d$ of various dimensions. We refer the readers to \cite{caprace2011rank}, \cite{bridson2013metric}, and \cite{charney2007introduction} for more details on these. One can assign many compact Hausdorff boundaries to a CAT$(0)$-cube complex (see, for example \cite[Section~1.3]{nevo2013poisson}). For our purposes, given a CAT$(0)$ metric space, we consider the action on the visual boundary $\partial X$ (see \cite[Chapter~8]{bridson2013metric}) equipped with the cone-topology. If $X$ is Gromov hyperbolic, then $\partial X$ is the classical Gromov boundary of $X$.

Let $\Gamma$ be a countable discrete group acting on a proper CAT$(0)$ cube complex $X$ (not necessarily hyperbolic) by isometries. We say that the action is elementary if the limit set $LX$ (the set of accumulation points in $\partial X$ of an orbit of the action) consists of at most two points or if $\Gamma$ fixes a point on $\partial X$.
\begin{example}[CAT$(0)$-cube complexes]
Let $\Gamma$ be a countable discrete group acting on a proper CAT$(0)$ cube complex $X$ (not necessarily hyperbolic) by isometries in a non-elementary way. Let $s\in\Gamma$ be rank-one isometry. It is well-known that any rank-one isometry $g\in\text{Isom}(X)$ has the north pole-south pole-dynamics (see, for example, \cite[Lemma~4.4]{hamenstadt2009rank}). Using \cite[Theorem~1.1]{hamenstadt2009rank}, we see that the limit set $LX\subset\partial X$ is perfect. It follows from \cite[Lemma~2.1]{boutonnet2021properly} that there is a non-atomic measure $\mu\in\text{Prob}(\partial X)$. Since $s$ is a rank-one isometry, there are two fixed points $x_s^+$ and $x_s^-$ on the visual boundary $\partial X$. Moreover, $s^ny\xrightarrow{n\to\infty} x_s^+$ for all $y\ne x_s^-\in \partial X$. Since $\mu$ is non-atomic, we see that $\mu(x_s^-)=0$. Now, let $(\mathcal{N},\Tilde{\tau})$ be a tracial von Neumann algebra, and $\Gamma\curvearrowright(\mathcal{N},\Tilde{\tau})$, a trace-preserving action. Setting $\mathcal{M}=\mathcal{N}\rtimes\Gamma$, it follows from \thref{tracialvna} that $L\left(\langle s \rangle\right)'\cap\mathcal{M}\le \mathcal{N}\rtimes\Gamma_{x_s^+}$. If we further assume that $\mathcal{N}$ is amenable, then in this case, since $\Gamma_{x_s^+}$ is amenable (see the argument in the last paragraph of \cite[Lemma~5.6]{ma2022boundary}), we conclude that $L\left(\langle s \rangle\right)'\cap\mathcal{M}$ is amenable. 
\end{example}
It is not difficult to find actions $\Gamma\curvearrowright X$ on CAT$(0)$-cube complexes $X$ which are non-elementary. For example, if $|\partial X|>2$ and $\Gamma\curvearrowright X$ cocompactly by isometries, then the action is necessarily non-elementary (see for example \cite{ballmann1995orbihedra}). 
\section{The case of \texorpdfstring{$\SL(d,\mathbb{\mathbb{Z}})$}{}}\label{sl}
This section applies our results to the von Neumann algebras associated with infinite subgroups of $\SL(d, \mathbb{Z})$, $d \geq 2$. We show that for each such subgroup $\Lambda$, the relative commutant of $L\Lambda$ is always contained in the von Neumann algebra of the intersection of some parabolic subgroup with $\SL(d,\mathbb{Z})$. In the case $d=3$ or if the subgroup is Zariski dense in $\SL(d,\mathbb{R})$, such parabolic subgroups are always Borel groups.
\begin{prop}\label{propa} Let $d \in \mathbb{N}$ and $\Gamma$ be an infinite subgroup of $\SL(d,\mathbb{Z})$. Then there is a parabolic subgroup $P$ of $\SL(d,\mathbb{R})$ such that $L\Gamma' \cap L \SL(d,\mathbb{Z}) \subset L(\SL (d,\mathbb{Z}) \cap P)$.
\end{prop}
\begin{proof} Let $G = \SL(d,\mathbb{R})$. We want to show that given a diverging sequence $(\gamma_n)$ in $\SL(d,\mathbb{Z})$ there are a parabolic subgroup $P$ of $G$, an $\SL(d,\mathbb{Z})$-quasi-invariant probability measure on $G/P$ and a point $y \in G/P$ such that, up to taking a subsequence, for $\mu$-almost every point $x$ in $G/P$ we have $\lim_n \gamma_n x= y$.\\ Let then $(\gamma_n)$ be such a sequence and write $\gamma_n =k_n a_n k_n'$ ($KAK$ decomposition in $\SL(d,\mathbb{R})$), in such a way that the diagonal entries $(\lambda_i^{(n)})$ of $a_n$ are taken in decreasing order: $\lambda_i(n) \geq \lambda_{i+1}^{(n)}$ for every $n$, for every $i =1,...,d$. Up to taking a subsequence we can suppose that $\lambda_i^{(n)} /\lambda_{i+1}^{(n)}$ converges to a point in $(0,\infty]$ for every $i=1,...,d-1$, and $k_n \rightarrow k$, $k_n' \rightarrow k'$ in $K$. We consider the partition of $\{1,...,d\}$ into $I_1,...,I_l$ subsets (for some $l \in \mathbb{N}$) defined by the condition that $i$ and $i+j$ belong to the same set $I_m$ if and only if $\lambda_i^{(n)}/\lambda_{i+j}^{(n)}$ converges to a finite number. Then we consider the parabolic subgroup $P$ associated with this partition, i.e., the one given by matrices in $\SL(d,\mathbb{R})$ of the form 
\begin{equation*}
\left(\begin{array}{ccccc}
 GL_{|I_1|}(\mathbb{R}) & * & * & ...& *\\
0 & GL_{|I_2|} (\mathbb{R}) & * & &\vdots\\
\vdots & \vdots & \vdots & \ddots&\vdots\\
0 & 0 & ...& 0 & GL_{|I_l|} (\mathbb{R})
\end{array}\right).
\end{equation*}
Let now $A= \{ g \in \SL(d,\mathbb{R}) \; | \; \det (g_i) \neq 0 \; \forall i=1,...,d-1\}$, where $g_i$ is the $i$-th principal minor of $g$. $A$ is a dense open subset of $G$. It follows from Gaussian elimination that every element of $A$ can be written as a product of an element of the group $T$ of strictly lower triangular matrices (i.e., the ones having only $1$'s on the diagonal) and an element from the Borel subgroup $B$ of upper triangular matrices (cf. the proof of \cite{Zi} Lemma 5.1.4). The map $T \rightarrow A/P$ is continuous and surjective; it restricts to a continuous surjective map $T \backslash \{T \cap P\} \rightarrow A/P \backslash \{eP\}$. Let then $x \in T \backslash \{T \cap P\}$ and write it as $x=(X_{i,j})_{i,j=1}^l$, where $X_{i,j}$ is a matrix of size $|I_i| \times |I_j|$; in the same way we write $a_n = (\Lambda_{i,j})_{i,j=1}^l$. Define the sequence in $P$ given by $h_n=(H_{i,j})_{i,j=1}^l$, where $H_{i,j} = \delta_{i,j} (\Lambda_{i,i} X_{i,i})^{-1}$. Then $a_n x h_n  \rightarrow e$ and so $a_n xP \rightarrow eP$. Let now $yP \in A/P$ and $C \subset A/P$ be a compact neighborhood of $yP$. Let $U$ be an open subset around $eP$ with an empty intersection with $C$. For every $xP \in C$ there is $n_{xP}$ such that $a_{n_xP} xP \in U$. Hence the open sets $a_{n_{xP}}^{-1} U$ cover $C$. It follows that the sequence $k^{-1} k_n a_n k_n' (k')^{-1} xP$ converges to $eP$. Hence, the result follows by choosing any $\SL(d,\mathbb{Z})$-quasi-invariant probability measure on $G/P$ which gives zero mass to $G/P \backslash A/P$. \end{proof}

\begin{prop}
    Let $n \geq 2$ and let $\Gamma= SL(n,\mathbb{Z})$. Let also $\mu$ be the $K$-invariant probability measure on the complete $n$-dimensional flag variety. Then, every subgroup of $\Gamma$, which is Zariski dense in $SL(n,\mathbb{R})$, contains a $\mu$-contracting sequence.
\end{prop}
\begin{proof}
    If $\Lambda$ is a Zariski dense subgroup of $\Gamma$, then we can apply the procedure in \cite[Theorem~3.6]{GoMa} (since the action of $\SL(n,\mathbb{Z})$ on the complete flag variety is transitive) to deduce that $\Lambda$ has the contraction property (as defined in \cite[Definition~3.1]{GoMa}. The result follows from \cite[Lemma~3.9]{GoMa} and the proof of Proposition \ref{propa}.\end{proof} 
\begin{cor}
    Let $\Lambda \subset \SL(d,\mathbb{\mathbb{Z}})$ be a Zariski-dense subgroup of $\SL(d,\mathbb{R})$. Then $L\Lambda' \cap L\SL(d,\mathbb{Z}) \subset L \Gamma_x$ for some $x \in \SL(d,\mathbb{R})/B_d$ (which is injective), where $B_d$ is the Borel subgroup of upper-triangular matrices in $\SL(d,\mathbb{R})$. In particular, this applies to infinite commensurated subgroups of $\SL(d,\mathbb{Z})$.
\end{cor}
\begin{proof} Follows from Theorem~\ref{tracialvna} and \cite[Lemma~7.5]{BaFu}.
\end{proof}
A stronger result holds if we assume $d=3$ in the above Proposition.
\begin{prop} Let $\Gamma$ be an infinite subgroup of $\SL(3,\mathbb{Z})$. Then $L\Gamma' \cap L\SL(3,\mathbb{Z}) \subset LB \cap \SL(3,\mathbb{Z})$ for some Borel subgroup $B \subset \SL(3,\mathbb{R})$. In particular, the relative commutant of every infinite subgroup of $\SL(3,\mathbb{Z})$ is injective.
\end{prop}
\begin{proof} It follows from the discussion in \cite[Example~7]{Oz} that every infinite subgroup of $\SL(3,\mathbb{Z})$ contains an element whose singular values are pairwise distinct. The result follows arguing as in the proof of Proposition \ref{propa}. \end{proof}

\begin{example}(\cite{boutonnet2021properly} Corollary 6.4) Let $\Lambda$ be an infinite subgroup of $P$, where $P$ is the parabolic subgroup associated to the partition $\{\{1,2\},\{3\}\}$, such that for every element $g \in \Lambda$, $1$ is a singular value of $g$. Then the relative commutant of every diffuse subalgebra of $L \Lambda$ inside $L\SL(3,\mathbb{Z})$ is injective. Indeed, by the proof of Proposition \ref{propa}, every divergent sequence in $\Lambda$ is $\mu$-contracting towards $eP$. It follows from Lemma \ref{lemb} and Theorem \ref{tracialvna} that the relative commutant of every diffuse von Neumann subalgebra of $L\Lambda$ is contained in $LP$, hence it coincides with its relative commutant inside $LP$. But $LP$ is solid, and the result follows. Note that if $\Lambda$ is contained in $\mathbb{Z}^2$ (after identifying $P$ with $\SL(2,\mathbb{Z}) \rtimes \mathbb{Z}^2$), the result follows from another application of Theorem \ref{tracialvna}.\end{example}
We can give a more general example in the case when $d
\ge 2$.
\begin{example} If $d \geq 2$ and $\Lambda$ is an infinite subgroup of $\SL(d,\mathbb{Z})$ with the property that for every $g \in \Lambda$, $g$ has only two singular values which are not $1$, then the relative commutant of every diffuse subalgebra of $L\Lambda$ is contained in the von Neumann algebra of the parabolic subgroup associated to the partition $\{\{1,...,d-1\}, \{d\}\}$. This is for example the case for certain embeddings of $\SL(2,\mathbb{Z})$ inside $\SL(d,\mathbb{Z})$ for $d \geq 2$.
\end{example}
\bibliographystyle{amsplain}
\bibliography{main}
\end{document}